\documentclass[reqno]{amsart}

\usepackage{setspace}
\usepackage{amsthm}
\usepackage{amsmath}
\usepackage{enumerate}
\usepackage{extarrows}
\usepackage{accents}

\usepackage{amssymb}

\newtheorem{teo}{Theorem}[section]

\newtheorem{coro}[teo]{Corollary}
\newtheorem{lema}[teo]{Lemma}

\newtheorem{prop}[teo]{Proposition}
\newtheorem{defi}[teo]{Definition}
\newtheorem{prob}[teo]{Problem}
\numberwithin{equation}{section}

\newcommand{\To}{\longrightarrow}

\title{Free sets for set-mappings relative to a family of sets}
\author{Antonio Avil\'{e}s}
\address{Universidad de Murcia, Departamento de Matem\'{a}ticas, Campus de Espinardo 30100 Murcia, Spain.} \email{avileslo@um.es}
\author{Claribet Pi\~{n}a}
\address{Universidad de los Andes, Departamento de Matem\'{a}aticas 1117111
Bogot\'{a}, Colombia.} \email{c.pina@uniandes.edu.co}

\thanks{First author supported by MINECO and FEDER (MTM2014-54182-P) and by Fundaci\'{o}n S\'{e}neca - Regi\'{o}n de Murcia (19275/PI/14).  Second author supported by Fondation Sciences Math\'{e}matiques de Paris.}

\subjclass[2010]{}


\begin{document}

\maketitle
\hyphenation{li-ne-ar}

\begin{abstract}
Given a family $\mathcal{F}$ of subsets of $\{1,\ldots,m\}$, we try to compute the least natural number $n$ such that for every function $S:[\aleph_n]^{<\omega}\To [\aleph_n]^{<\omega}$ there exists a bijection $u:\{1,\ldots,m\}\To Y\subset \aleph_n$ such that $Su(A)\cap Y \subset u(A)$ for all $A\in\mathcal{F}$.\\
\end{abstract}

\emph{Added in proof:  The answer to Problem 1 in page 2 is YES under GCH, but NO in other models for $m>3$ (cf. [5]). Thus, Problem 3.6 should be considered under GCH.}

\section{Introduction}

For a set $X$, $[X]^{<\omega}$ denotes the family of all finite subsets of $X$, and, for a cardinal $n$, $[X]^n$ denotes the family of all subsets of $X$ of cardinality $n$. The following is a variant of a classical result of Kuratowski and Sierpi\'{n}ski \cite{Kuratowski}:

\begin{teo}\label{classical}
For a set $X$ and a natural number $n$ the following are equivalent:
\begin{enumerate}
\item $|X|\leq \aleph_{n-1}$,
\item For every function $S:[X]^{<\omega}\To [X]^{<\omega}$ there exists a set $Y\in [X]^n$ such that $S(A)\cap Y \subset A$ for all $A\subset Y$.
\end{enumerate}
\end{teo}

In this context, the set $Y$ obtained is said to be a free set.
This combinatorial principle has some applications in the theory
of nonseparable Banach spaces \cite{EnfRos} concerning bases in
spaces of integrable functions. For another application in Banach
space theory, this time about extension operators, a similar
principle was used in \cite{AntonioWitold2015} in which $X$ can be
of cardinality $\aleph_1$, $Y$ can be found of arbitrarily large
finite cardinality, but the property does not hold for all but
only for some subsets $A\subset Y$.

\begin{teo}\label{AviMar}
Let $X$ be an uncountable set, $m$ a natural number, and
$S:[X]^{<\omega}\To [X]^{<\omega}$ a function. Then, there exists
$Y\subset X$ and a bijection $u:\{1,\ldots,m\}\To Y$ such that
$Su(A)\cap Y \subset u(A)$ whenever $A$ is of the form $A = \{k :
i\leq k \leq j\}$ for some $i,j\in \{1,\ldots,m\}$.
\end{teo}

This is not a complicated result, but it points in a different direction than most of the research that has been done in this topic of free sets for set-mappings, cf. for instance \cite{ErdHaj,HajMat,KomShe}. Instead of looking at the cardinality of the sets involved, we may look at the structure of the family of sets that freedom refers to.

\begin{defi}\label{maindef}
Let $\mathcal{F}$ be a finite family of finite sets. We define the natural number $\mathfrak{fr}(\mathcal{F})$ as the minimum $n<\omega$ such that for every function $S:[\aleph_n]^{<\omega}\To [\aleph_n]^{<\omega}$  there exists $Y\subset \aleph_n$ and a bijection $u:\bigcup\mathcal{F}\To Y$ such that $Su(A) \cap Y \subset u(A)$ for all $A\in \mathcal{F}$.
\end{defi}

Theorem~\ref{classical} implies that $\mathfrak{fr}(\mathcal{F})$ is a well defined natural number for every finite family $\mathcal{F}$ of finite sets. The two previous results can be now restated as:

\begin{enumerate}
\item If $\mathcal{F}$ is the family of all subsets of $\{1,\ldots,n\}$, then $\mathfrak{fr}(\mathcal{F}) = n-1$.

\item If $\mathcal{F}$ is the family of all  intervals of $\{1,\ldots,n\}$, then $\mathfrak{fr}(\mathcal{F}) = 1$.
\end{enumerate}

The general problem that we address in this paper is to compute $\mathfrak{fr}(\mathcal{F})$ for any family $\mathcal{F}$. We shall give some partial results when $\mathcal{F}$ is of certain specific forms. However, we are far from a satisfactory understanding of this function, and we are unable to compute it even in some specific simple cases. The following are open problems for us:\\

\begin{enumerate}

\item Is $\mathfrak{fr}([X]^m) = m$ for every finite $X$ and every natural number $m$ with $m<|X|$? In the simplest case unknown to us, is $\mathfrak{fr}[\{1,2,3,4,5\}]^2 = 2$? \\

\item  Is $\mathfrak{fr}(12,23,34,35,56,123,235,356,2356) = 1$? (This is family of subsets of $\{1,2,3,4,5,6\}$ in simplified notation).\\

\item Can we compute $\mathfrak{fr}$ with the recursive formula $$\mathfrak{fr}(\mathcal{F}) = 1 + \max_{A\subset X} \ \min_{a\in A} \ \mathfrak{fr}\{B\cap A \setminus\{a\} : a\in B\in \mathcal{F}\} \ ?$$
Proving the inequality $[\leq]$ would imply the equality, by results of Section~\ref{secnes}.
\end{enumerate}

\section{Elementary properties of $\mathfrak{fr}(\mathcal{F})$}\label{sectionproperties}

In the definition of $\mathfrak{fr}(\mathcal{F})$, we can restrict ourselves to only certain special functions $S$:

\begin{lema}\label{specialS}
Let $\mathcal{F}$ be a finite family of finite sets and $n\geq 1$. Suppose that for every function $S:[\aleph_n]^{<\omega}\To [\aleph_n]^{<\omega}$ satisfying
\begin{itemize}
\item $S(A)\subset S(B)$ whenever $A\subset B$,
\item $\alpha\not\in S(A)$ whenever $\max(A) < \alpha\in\aleph_n$,
\item $S(\{\alpha\}) = \emptyset$ for all $\alpha<\aleph_n$,
\end{itemize}
 there exists $Y\subset \aleph_n$ and a bijection $u:\bigcup\mathcal{F}\To Y$ such that $Su(A) \cap Y \subset u(A)$ for all $A\in \mathcal{F}$. Then $\mathfrak{fr}(\mathcal{F}) \leq n$.
\end{lema}

\begin{proof}
Let $S:[\aleph_n]^{<\omega}\To [\aleph_n]^{<\omega}$ be an arbitrary function. By Fodor's pressing down lemma, there exists $\gamma_1<\aleph_n$ and a stationary subset $Z_1\subset \aleph_n$ such that $\max\{ \gamma<\alpha : \gamma\in S(\{\alpha\})\} < \gamma_1$ for all $\alpha\in Z_1$. Let $Z_2 = \{\alpha\in Z_1 : \alpha>\gamma_1\}$. Define by induction an increasing function $f:Z_2\To Z_2$ such that $f(\alpha) > \max(S(B))$ whenever $B\in [Z_2]^{<\omega}$ and $\max(B)<\alpha$. Let $Z_3$ be the range of $f$. The new function $\tilde{S}:[Z_3]^{<\omega}\To [Z_3]^{<\omega}$ defined by $\tilde{S}(A) = Z_3\cap \bigcup_{B\subset A} S(B)$ satisfies the three conditions of the lemma (when we identify $Z_3$ with $\aleph_n$ via a monotone enumeration), so we get $Y$ and $u$, that witness that $\mathfrak{fr}(\mathcal{F})\leq n$ as in Definition~\ref{maindef}.
\end{proof}

\begin{lema}\label{case0}
$\mathfrak{fr}(\mathcal{F}) = 0$ if and only if for all $A,B\in\mathcal{F}$, either $A\subset B$ or $B\subset A$.
\end{lema}

\begin{proof}
Suppose that $\mathfrak{fr}(\mathcal{F}) = 0$ and consider $S:[\aleph_0]^{<\omega}\To [\aleph_0]^{<\omega}$ given by $S(A) = \{0,1,\ldots,\max(A)\}$. If we apply Definition~\ref{maindef}, we get $Y$ and $u$ such that $Su(A)\cap Y\subset u(A)$ for all $A\in \mathcal{F}$. This implies that $u(A)\cap Y$ is an initial segment of $Y$ for all $A\in \mathcal{F}$. Hence $\mathcal{F}$ is linearly ordered by inclusion.  Conversely, suppose that $\mathcal{F}$ is linearly ordered by inclusion. In fact, we can suppose that $\mathcal{F} = \{ \emptyset, \{1\}, \{1,2\},\ldots,\{1,2,3,\ldots\}\}$. Given $S:[\aleph_0]^{<\omega}\To [\aleph_0]^{<\omega}$, we can define $u$ inductively, taking care that $u(k)$ is larger than the maximum of $\bigcup \{u(A) : \max(A)<k\}$.
\end{proof}

The index is monotone and it does not grow by adding intersections:

\begin{lema}\label{intersections}
Let $\mathcal{F}$ and $\mathcal{F}'$ be two finite families of fine sets such that each element of $\mathcal{F}$ can be expressed as an intersection of elements of $\mathcal{F}'$. Then $\mathfrak{fr}(\mathcal{F})\leq \mathfrak{fr}(\mathcal{F}')$.
\end{lema}

\begin{proof}
Let $n = \mathfrak{fr}(\mathcal{F}')$, and let $S:[\aleph_n]^{<\omega}\To [\aleph_n]^{<\omega}$ be a function. Let $S_1:[\aleph_n]^{<\omega}\To [\aleph_n]^{<\omega}$ be given by $S_1(A) = \bigcup_{B\subset A}S(A)$. There exists $Y\subset \aleph_n$ and a bijection $u:\bigcup\mathcal{F}'\To Y$ such that $S_1u(A) \cap Y \subset u(A)$ for all $A\in \mathcal{F}'$. Now, if $B\in \mathcal{F}$, it is of the form $B=A_1\cap \cdots \cap A_m$ with $A_i\in \mathcal{F}'$ and
$$Su(B) \cap Y \subset S_1u(A_i) \cap Y \subset u(A_i)$$
for every $i$, therefore $Su(B) \cap Y \subset \bigcap_i u(A_i) = u(B)$.
\end{proof}

It neither grows by taking restrictions.

\begin{lema}\label{restrictions}
Let $\mathcal{F}$ be a family of subsets of a finite set $X$. Consider $Y\subset X$ and $\mathcal{F}|_Y = \{A\cap Y : A\in\mathcal{F}\}$. Then $\mathfrak{fr}(\mathcal{F}|_Y) \leq \mathfrak{fr}(\mathcal{F})$.
\end{lema}

\begin{proof}
Let $n=\mathfrak{fr}(\mathcal{F})$. Fix $S:[\aleph_n]^{<\omega}\To [\aleph_n]^{<\omega}$ as in Lemma~\ref{specialS}. Then we have $u:X\To \aleph_n$ satisfying $Su(A)\cap u(X)\subset u(A)$ for all $A\in \mathcal{F}$. But this implies that $Su(A\cap Y) \cap u(Y) \subset Su(A)\cap u(Y) \subset u(A)\cap u(Y) = u(A\cap Y)$ for all $A\in\mathcal{F}$.
\end{proof}

Singletons and the total set are irrelevant:

\begin{lema}\label{singletons}
If $\mathcal{F}$ is a family of subsets of $X$, then
$$\mathfrak{fr}(\mathcal{F} \cup \{X\}\cup \{\{x\} : x\in X\})= \max(1,\mathfrak{fr}(\mathcal{F})).$$
\end{lema}

\begin{proof}
The fact that adding $X$ to the family is meaningless follows directly from the definition. For the singletons, we use Lemma~\ref{specialS} and the fact that we can suppose that $S$ satisfies $S(\{\alpha\}) = \emptyset$ for all $\alpha$.
\end{proof}

\section{Nested orders}\label{secnes}

There are certain families $\mathcal{F}$ that seem critical for this problem.

\begin{defi}\label{defnest}
Let $X$ be a set and $\mathfrak{S}$ be a family of finite sequences of elements of $X$. We say that $\mathfrak{S}$ is a nested-orders family on $X$ if the following hold:
\begin{enumerate}
\item $(t)\in\mathfrak{S}$ for all $t\in X$.
\item If $(t_1,\ldots,t_k)\in \mathfrak{S}$, then $t_i\neq t_j$ for all $i\neq j$.
\item If $(t_1,\ldots,t_k)\in \mathfrak{S}$, then $(t_1,\ldots,t_{k-1})\in\mathfrak{S}$ and $(t_1,\ldots,t_{k-2},t_k)\in\mathfrak{S}$.
\item If $(t_1,\ldots,t_k,s,t)\in \mathfrak{S}$ and $(t_1,\ldots,t_k,t,u)\in \mathfrak{S}$, then $(t_1,\ldots,t_k,s,u)\in \mathfrak{S}$.
\item If $(t_1,\ldots,t_k,t)\in \mathfrak{S}$, $(t_1,\ldots,t_k,s)\in \mathfrak{S}$ and $t\neq s$, then eiher $(t_1,\ldots,t_k,t,s)\in \mathfrak{S}$ or $(t_1,\ldots,t_k,s,t)\in \mathfrak{S}$.
\end{enumerate}
\end{defi}

It will be convenient to consider also the restricted version of this definition:

\begin{defi}\label{defnestn}
Let $n$ be a natural number, let $X$ be a set and $\mathfrak{S}$
be a family of finite sequences of length at most $n+2$ of
elements of $X$. We say that $\mathfrak{S}$ is an
$n$-nested-orders family on $X$ if the conditions (1), (2), (3)
and (4) in Definition~\ref{defnest} hold, while condition (5)
holds for $k<n$.
\end{defi}

For such a family $\mathfrak{S}$ as above and a natural number $n$, we consider the family of sets
$$\mathcal{F}_{\mathfrak{S},n} = \{A\subset X : (t_1,\ldots,t_{n+1},t_{n+2})\in \mathfrak{S}, t_1,\ldots,t_{n},t_{n+1}\in A \Rightarrow t_{n+2}\in A\}.$$

Observe that $\mathcal{F}_{\mathfrak{S},m} \subset \mathcal{F}_{\mathfrak{S},n}$ if $m<n$. The name of nested orders is because if we have $u=(t_1,\ldots,t_k)\in \mathfrak{S}$, then the relation given by $t\prec_u s$ if $(t_1,\ldots,t_k,s,t)\in \mathfrak{S}$ is a linear order on the set $\{t : (t_1,\ldots,t_k,t)\in \mathfrak{S}\}$, and this set is in turn the initial segment below $t_k$ in the order relation $\prec_v$ determined by $v=(t_1,\ldots,t_{k-1})$. So $\mathfrak{S}$ comes from a family of linear order relations on nested initial segments. Conversely, a nested-orders family can always be constructed in the following way: Begin with a linear order on $X$ and declare that $(t_1,t_2)\in \mathfrak{S}$ if and only if $t_1>t_0$. Given the sequences of $\mathfrak{S}$ of lenght at most $m$, for every $t = (t_1,\ldots,t_{m-2})\in\mathfrak{S}$, choose a linear order $\prec_t$ on $\{t : (t_1,\ldots,t_{m-2},t)\in \mathfrak{S}\}$, and then declare that $(t_1,\ldots,t_m)\in \mathfrak{S}$ if and only if $t_m \prec_t t_{m-1}$. This argument is also useful for the following

\begin{lema}\label{nestedversions}
For a natural number $n$ and a family $\mathcal{F}$ of subsets of a set $X$, the following are equivalent:
\begin{enumerate}
\item There exists a nested-orders family $\mathfrak{S}$ such that $\mathcal{F}\subset \mathcal{F}_{\mathfrak{S},n}$.
\item There exists a nested-orders family $\mathfrak{S}$ such that $\mathcal{F}\subset \mathcal{F}_{\mathfrak{S},m}$ for all $m\geq n$.
\item There exists an $n$-nested-orders family $\mathfrak{S}$ such that $\mathcal{F}\subset \mathcal{F}_{\mathfrak{S},n}$.
\end{enumerate}
\end{lema}

\begin{proof}
$[2\Rightarrow 1]$ is obvious. That $[1\Rightarrow 2]$ follows from the observation that if $\mathfrak{S}$ is nested-orders, then $\mathcal{F}_{\mathfrak{S},n} \subset \mathcal{F}_{\mathfrak{S},n+1}$, which follows easily from the properties of the definition. That $[1\Rightarrow 3]$ is clear because if $\mathfrak{S}$ is a nested-orders family, the sequences of $\mathfrak{S}$ of length at most $n+2$ form an $n$-nested-orders family. For the converse $[3\Rightarrow 1]$, if $\mathfrak{S}$ is an $n$-nested-orders family, we can enlarge it to a nested-orders family, by adding sequences of larger length inductively, as described in the comments just before Lemma~\ref{nestedversions}
\end{proof}

 There is a natural way to construct nested-orders families of sequences of length at most $n+1$ of ordinals below $\aleph_n$, and using that one can prove:

\begin{teo}\label{reorderings}
Let $\mathcal{F}$ be a family of subsets of a finite set $X$ such that  $\mathfrak{fr}(\mathcal{F}) \leq n$. Then there exists an $n$-nested-orders family $\mathfrak{S}$ such that $\mathcal{F}\subset \mathcal{F}_{\mathfrak{S},n}$.
\end{teo}

\begin{proof}
We are going to construct an $n$-nested-orders family $\mathfrak{S}_n$ of finite sequences of ordinals below $\aleph_n$ with the extra property that
$$ \left|\{t<\aleph_n : (t_1,\ldots,t_k,t) \in\mathfrak{S}_n\} \right| < \aleph_{n+1-k} $$
for each $(t_1,\ldots,t_k)$ with $k\leq n+1$. In particular, $\{t<\aleph_n : (t_1,\ldots,t_{n+1},t) \in\mathfrak{S}_n\}$ will be finite for any $t_1,\ldots,t_{n+1}<\aleph_n$, and we can define a function $S:[\aleph_n]^{<\omega}\To [\aleph_n]^{<\omega}$ by
$$S(A) = \{ t<\aleph_n : \exists t_1,\ldots,t_{n+1}\in A : (t_1,\ldots,t_{n+1},t)\in \mathfrak{S}_n\}.$$
If $\mathfrak{fr}(\mathcal{F})\leq n$ then we have the corresponding bijection $u:X\To Y\subset \aleph_n$. The family $\mathfrak{S} = \{ (t_1,\ldots,t_k) : (u(t_1),\ldots,u(t_k))\in \mathfrak{S}_n\}$ would be the desired family of finite sequences. It remains to construct $\mathfrak{S}_n$. This is a standard recursive enumeration procedure. We declare, inductively on $k$, which sequences $(t_1,\ldots,t_k)$ belong to $\mathfrak{S}_n$. The empty sequence, and the one-element sequences $(t_1)$ with $t_1<\aleph_n$ they all belong to $\mathfrak{S}_n$. A sequence $(t_1,t_2)$ belongs to $\mathfrak{S}_n$ if and only if $\aleph_n>t_1>t_2$. Suppose that we already know which sequences $(t_1,\ldots,t_k)$ belong to $\mathfrak{S}_n$ and we want to declare which sequences of lenght $k+1$ belong. We assume in the inductive hypothesis that
$$ \left|\{t<\aleph_n : (t_1,\ldots,t_{k-1},t) \in\mathfrak{S}_n\} \right| < \aleph_{n+2-k} .$$
Then we can find an injection $$\phi_{(t_1,\ldots,t_{k-1})}: \{t<\aleph_n : (t_1,\ldots,t_{k-1},t) \in\mathfrak{S}_n\} \To \aleph_{n+1-k}$$ and we can declare that $(t_1,\ldots,t_k,t_{k+1})\in \mathfrak{S}_n$ if and only if $(t_1,\ldots,t_{k-1},t_k)\in \mathfrak{S}_n$, $(t_1,\ldots,t_{k-1},t_{k+1})\in \mathfrak{S}_n$ and
$\phi_{(t_1,\ldots,t_{k-1})}(t_k) > \phi_{(t_1,\ldots,t_{k-1})}(t_{k+1})$.
\end{proof}

We can define a new index $\mathfrak{no}(\mathcal{F})$ as the least integer $n$ such that there exists a nested-orders family $\mathfrak{S}$ of degree $n$ such that $\mathcal{F}\subset \mathcal{F}_{\mathfrak{S},n}$.

\begin{coro}\label{nolessfr}
$\mathfrak{no}(\mathcal{F}) \leq \mathfrak{fr}(\mathcal{F})$ for every $\mathcal{F}$.
\end{coro}

But the following is an open problem for us:

\begin{prob}\label{problemno}
Is $\mathfrak{no}(\mathcal{F}) = \mathfrak{fr}(\mathcal{F})$ for every $\mathcal{F}$?
\end{prob}

In other words, the question is whether $\mathfrak{fr}(\mathcal{F}_{\mathfrak{S},n})\leq n$ for any finite nested-orders family $\mathfrak{S}$ of degree $n$. Along the rest of this section, $X$ is a fixed finite set and $\mathcal{F}$ is a family of subsets of $X$.

\begin{lema}\label{nestbound}
For every integer $k\leq |X|-\mathfrak{no}(\mathcal{F})$ we have
$$\left|\mathcal{F} \cap [X]^{\mathfrak{no}(\mathcal{F})+k}\right| \leq \binom{\mathfrak{no}(\mathcal{F})}{|X|-k}.$$
\end{lema}

\begin{proof}
Let $n=\mathfrak{no}(\mathcal{F})$ and assume that $\mathcal{F} = \mathcal{F}_{\mathfrak{S},n}$ for some nested-orders family of degree $n$. Let $A\in \mathcal{F}_{\mathfrak{S},n} \cap [X]^{n+k}$. By induction, pick $t_1,\ldots,t_{n+1}\in A$ in such a way $t_j$ is the $\prec_{(t_1,\ldots,t_{j-1})}$-maximum of $A\setminus \{t_1,\ldots,t_{j-1}\}$. At the end, since $A\in \mathcal{F}_{\mathfrak{S},n}$, we must have that
$$A = \{t_1,\ldots, t_{n+1}\} \cup \{t\in X : t \prec_{(t_1,\ldots,t_{n})} t_{n+1}\}.$$
Thus, everything reduces to count how many sequences $(t_1,\ldots,t_{n+1})$ are there such that $t_i \prec_{(t_1,\ldots,t_{j-1})} t_j$ for all $i>j$ and such that $|\{t\in X : t \prec_{(t_1,\ldots,t_{n})} t_{n+1}\}|=k-1$. To each such sequence we can associate the numbers $m_i = |\{t\in X: t\preceq_{(t_1,\ldots,t_{i-1})} t_i\}|$. These numbers satisfy $|X|\geq m_1 > m_2  > m_3 > \cdots > m_{n+1} = k$
\end{proof}

\begin{coro}
For every integer $k\leq |X|-\mathfrak{fr}(\mathcal{F})$ we have
$$\left|\mathcal{F} \cap [X]^{\mathfrak{fr}(\mathcal{F})+k}\right| \leq \binom{\mathfrak{fr}(\mathcal{F})}{|X|-k}.$$
\end{coro}

\begin{proof}
Write $\mathfrak{fr}(\mathcal{F}) = \mathfrak{no}(\mathcal{F}) + i$ with $i\geq 0$. By Lemma~\ref{nestbound},
$$\left|\mathcal{F} \cap [X]^{\mathfrak{fr}(\mathcal{F})+k}\right| = \left|\mathcal{F} \cap [X]^{\mathfrak{no}(\mathcal{F})+i+k}\right|$$ $$\leq \binom{\mathfrak{no}(\mathcal{F})}{|X|-(k+i)} = \binom{\mathfrak{fr}(\mathcal{F})-i}{|X|-k-i} \leq \binom{\mathfrak{fr}(\mathcal{F})}{|X|-k} .$$
\end{proof}

%

\begin{lema}\label{norestriction}
For every subset $B\subset X$ we have $\mathfrak{no}(\mathcal{F}|_B) \leq  \mathfrak{no}(\mathcal{F})$.
\end{lema}

\begin{proof}
Suppose that $\mathcal{F}\subset \mathcal{F}_{\mathfrak{S},n}$, and define $\mathfrak{S}^B =  \{(t_1,\ldots,t_k) \in \mathfrak{S} : t_1,\ldots,t_k\in B\}$. It is easy to check that $\mathfrak{S}^B$ is a nested-orders family on $B$ and $\mathcal{F}|_B \subset \mathcal{F}_{\mathfrak{S}^B,n}$
\end{proof}

\begin{lema}\label{recursive}
The following recursive formula for the computation of $\mathfrak{no}$ holds: $$\mathfrak{no}(\mathcal{F}) = 1 + \max_{A\subset X} \ \min_{a\in A} \ \mathfrak{no}\{B\cap A \setminus\{a\} : a\in B\in \mathcal{F}\}$$

\end{lema}

\begin{proof}
By Lemma~\ref{norestriction}, for the inequality $[\geq]$, it is enough to check the general inequality
$$\mathfrak{no}(\mathcal{F}) >  \min_{a\in X} \ \mathfrak{no}\{B\setminus\{a\} : a\in B\in \mathcal{F}\}.$$
So we suppose that $\mathcal{F} \subset \mathcal{F}_{\mathfrak{S},n}$ and we find $a\in X$ such that $\mathfrak{no}\{B\setminus\{a\} : a\in B\in\mathcal{F} \} \leq n-1$. Pick $a$ the maximum of $X$ in the order $\prec_\emptyset$ given by $t\prec_\emptyset s$ iff $(s,t)\in \mathfrak{S}$. If $a\in B \in \mathcal{F}_{\mathfrak{S},n}$, then $B\setminus\{a\}\in \mathcal{F}_{\mathfrak{S}',n-1}$, where
$$\mathfrak{S}' = \{(t_1,\ldots,t_p) : (a,t_1,\ldots,t_p)\in\mathfrak{S}\} $$
is a nested-orders family.
For the other inequality $[\leq]$ we call $n$ to the right-hand side of the equation. We need to define a suitable nested-orders family $\mathfrak{S}$ such that $\mathcal{F}\subset \mathcal{F}_{\mathfrak{S},n}$. By Lemma~\ref{nestedversions}, for every $A\subset X$ we can find $a[A]$ and a nested-orders family $\mathfrak{S}[A]$ such that $$\{B\cap A \setminus\{a[A]\} : a[A]\in B\in \mathcal{F}\} \subset \mathcal{F}_{\mathfrak{S}[A],n-1}.$$
Define $a_1 = a[X]$, $a_2 = a[X\setminus \{a_1\}]$, $a_3 = a[X\setminus\{a_1,a_2\}]$, etc, so that $X = \{a_1,a_2,\ldots,a_m\}$. Let $A_k = \{a_k,a_{k+1},\ldots,a_m\}$, so that $a_k = a[A_k]$.
We define $\mathfrak{S}$ to be the family of all finite sequences of the form $(a_k,t_2,\ldots,t_p)$ such that $t_2,\ldots,t_p\in A_k\setminus \{a_k\}$ and $(t_2,\ldots,t_p)\in \mathfrak{S}[A_k]$. This is easily checked to be a nested-orders family on $X$. Now, we pick $B\in \mathcal{F}$ and we check that $B\in \mathcal{F}_{\mathfrak{S},n}$. So suppose that $(a_k,t_2,\ldots,t_{n+2})\in \mathfrak{S}$ and $a_k,t_2,\ldots,t_{n+1}\in B$. Then $a[A_k] = a_k \in B\in \mathcal{F}$, so $B\cap A_k \setminus \{a_k\}\in \mathcal{F}_{\mathfrak{S}[A_k],n-1}$. Since $(t_2,\ldots,t_{n+1})\in B\cap A_k \setminus\{a_k\}$, we conclude that $t_{n+2}\in B\cap A_k\setminus\{a_k\}$, and in particular $t_{n+2}\in B$. This shows that $B\in \mathcal{F}_{\mathfrak{S},n}$.
\end{proof}

\begin{lema}\label{no0}
$\mathfrak{no}(\mathcal{F}) = 0$ if and only if $\mathfrak{fr}(\mathcal{F}) = 0$ if and only if $\mathcal{F}$ is linearly ordered by inclusion.
\end{lema}

\begin{proof}
Observe that a family of the form $\mathcal{F}_{\mathfrak{S},0}$ consists of the initial segments of the order $\prec_\emptyset$, so it is linearly ordered by inclusion. Using Lemma~\ref{case0} and Corollary~ \ref{nolessfr}, we are done.
\end{proof}

Thus, Lemma~\ref{recursive} describes the index $\mathfrak{no}$ in a recursive way starting from the families linearly ordered by inclusion which are those of index 0. Using this, we easily compute the value of $\mathfrak{no}$ for the family $[X]^{<\omega}$ of all subsets of $X$:

\begin{prop}\label{noall}
$\mathfrak{no}([X]^{<\omega}) = |X|-1$.
\end{prop}

As a corollary of Lemma~\ref{recursive}, and Lemma~\ref{no0}, Problem~\ref{problemno} is equivalent to Problem 3 in the introduction (indeed, by Corollary~\ref{nolessfr}, equivalent to asking for the $[\leq]$ inequality).

\section{Intersections of initial segments}\label{secseg}

\begin{teo}\label{constantorders}
Suppose that we have $n+1$ many linear order relations $<_1,<_2,\ldots,<_{n+1}$ on a finite set $X$. Then the family
$$\mathcal{F} = \left\{A\subset X : \exists s_1,\ldots,s_{n+1}\in X : A = \{t\in X : t\leq_1 s_1, \ldots, t \leq_{n+1} s_{n+1}\}\right\}$$ satisfies that $\mathfrak{fr}(\mathcal{F}) \leq n$.
\end{teo}

\begin{proof}
Let $S:[\aleph_n]^{<\omega}\To [\aleph_n]^{<\omega}$ be a function. We shall define sets $A_{ix}\subset \aleph_n$ for $i\in \{0,1,\ldots,n+1\}$ and $x\in X$ with the following properties:
\begin{enumerate}
\item $A_{0x} = \aleph_n$ for all $x$,
\item $A_{1x}\cap A_{1x} = \emptyset$ if $x\neq y$,
\item $A_{ix} \supset A_{jy}$ if $i<j$,
\item $|A_{ix}| = \aleph_{n-i}$ for $i=1,\ldots,n$ and $A_{n+1,x}$ is a singleton $\{u(x)\}$ for all $x\in X$,
\item $S(A_{ix})\cap A_{iy} = \emptyset$ if $x <_i y$.
\end{enumerate}
These sets are constructed by induction on $i$. Once $i$ is fixed, if $X = \{x_1 <_i x_2 <_i \cdots <_i x_m\}$ one easily constructs each $A_{i x_k}$ by induction on $k$, one just takes $A_{i x_k}\subset A_{i-1,x_k}$ such that $A_{i x_k} \cap \bigcup_{p<k} S(A_{i,x_p}) = \emptyset$, which is possible since $|A_{i-1,x_k}| = \aleph_{n-i+1} > |\bigcup_{p<k} S(A_{i,x_p})|$. The set $Y = u(X)$ and the bijection $u:X\To Y$ are the ones that we are looking for, because (5) guarantees that if $A\subset X$ and $x <_i y$ for all $x\in A$, then $u(y)\not\in S(u(A))$.
\end{proof}

Notice that Theorem~\ref{AviMar} is a corollary of Theorem~\ref{constantorders} obtained by taking $<_1$ the natural order of $\{1,\ldots,m\}$ and $<_2$ the reverse order. We can also obtain Theorem~\ref{classical}:

\begin{prop}\label{classicalformula}
$\mathfrak{fr}([X]^{<\omega}) = |X|-1$.
\end{prop}

\begin{proof}
We consider linear orders $<_1,\ldots,<_{n+1}$ on a set $X=\{s_1,\ldots,s_{n+1}\}$ of cardinality $n+1$ such that the maximum of $<_k$ is $s_k$. In fact, any subset $A\subset X$ can be expressed as intersections of initial segments of the form $\{x : x\leq s_k\}$ if $s_k\in A$ or $\{x : x<_k s_k\}$ if $s_k\not\in A$. This proves the inequality $[\leq]$. The other inequality follows from \ref{noall} and \ref{nolessfr}.
\end{proof}

\section{Families of subsets of $\{1,2,3,4\}$}

\begin{lema}\label{proper}
If $\mathcal{F}$ is a family of subsets of $X$ which does not contain all subsets of $X$ of cardinality $|X|-1$, then $\mathfrak{fr}(\mathcal{F})\leq |X|-2$.
\end{lema}

\begin{proof}
Say that $X=\{1,\ldots,n\}$ and $\{1,\ldots,n-1\}\not\in
\mathcal{F}$. For every $k\in\{1,\ldots,n-1\}$ consider a linear
order $<_k$ on $X$ whose last two elements are $n <_k k$. All
subsets of $X$ except $\{1,\ldots,n-1\}$ can be written as
intersections of initial segments of these orders, so we can apply
Lemma~\ref{constantorders}. In fact, if $n\in A\subseteq X$, then
$$A = \bigcap_{i\in A\setminus\{n\}}\{x\in X : x\leq_ i i\} \cap \bigcap_{i\in \{1,\ldots,n-1\}\setminus A} \{x\in X : x \leq_i n\}$$
while, if $n\not \in A$, $A\neq \{1,\ldots,n-1\}$, then
$$A = \bigcap_{i\in A}\{x\in X : x\leq_ i i\} \cap \bigcap_{i\in \{1,\ldots,n-1\}\setminus A} \{x\in X : x <_i n\}$$
\end{proof}

If $\mathcal{F}$ is a family of subsets of $X$, we say that $\mathcal{F}$ contains a cycle if there exists a set $A = \{x_1,\ldots,x_k\}\subset [X]$ of cardinality at least 3 such that
$$\{x_1,x_2\}, \{x_2,x_3\},\ldots, \{x_{k-1},x_k\},\{x_k,x_1\} \in \mathcal{F}|_A.$$

\begin{lema}\label{cycle}
If $\mathcal{F}$ contains a cycle, then $\mathfrak{no}(\mathcal{F})>1$. Hence, also $\mathfrak{fr}(\mathcal{F})>1$.
\end{lema}

\begin{proof}
It follows from Lemma~\ref{nestbound}, because we can find no $a\in A$ such that $\mathfrak{no}\{B\cap A\setminus\{a\} : a\in B\in \mathcal{F}\} = 0$, that is, such that $\{B\cap A\setminus\{a\} : a\in B\in \mathcal{F}\}$ is linearly ordered by inclusion.
\end{proof}

\begin{teo}
If $\mathcal{F}$ is a family of subsets of $X = \{1,2,3,4\}$, then $\mathfrak{fr}(\mathcal{F}) = \mathfrak{no}(\mathcal{F})$.
More precisely, supposing that $\mathcal{F}$ is closed under intersections,
\begin{enumerate}
\item $\mathfrak{fr}(\mathcal{F})= \mathfrak{no}(\mathcal{F} ) = 0$ if and only if $\mathcal{F}$ is linearly ordered by inclusion.
\item $\mathfrak{fr}(\mathcal{F}) = \mathfrak{no}(\mathcal{F} ) = 3$ if and only if $\mathcal{F} =[X]^{<\omega}$.
\item $\mathfrak{fr}(\mathcal{F}) = \mathfrak{no}(\mathcal{F} ) = 2$ if and only if $\mathcal{F} \neq [X]^{<\omega}$ and $\mathcal{F}$ contains a cycle.
\item $\mathfrak{fr}(\mathcal{F}) = \mathfrak{no}(\mathcal{F} ) = 1$ in the remaining cases.
\end{enumerate}
\end{teo}

\begin{proof}
Statement (2) follows from~\ref{no0}. The implication $[\Leftarrow]$ of (1) follows from~\ref{classicalformula} and~\ref{noall}. The converse follows from Lemma~\ref{proper}, since we are assuming (for simplicity, based on Lemma~\ref{intersections}) that $\mathcal{F}$ is closed under intersections, so if $\mathcal{F}\neq [X]^{<\omega}$ it must fail to contain a set of cardinality $3$ (we assume $X\in \mathcal{F}$ as $X$ is the result of intersecting an empty family). The implication $[\Leftarrow]$ of (3) follows from (2) and Lemma~\ref{cycle}. It only remains to show that if $\mathcal{F}\neq [X]^{<\omega}$, and $\mathcal{F}$ does not contain a cycle, then $\mathfrak{fr}(\mathcal{F}) \leq 1$. Let $\mathcal{F}^\ast$ be the elements of $\mathcal{F}$ of cardinality either 2 or 3. By Lemma~\ref{singletons}, $\mathfrak{fr}(\mathcal{F}) \leq \max(\mathfrak{fr}(\mathcal{F}^\ast),1)$.  If $\mathcal{F}$ contains three elements of cardinality 3, say $\{1,2,3\}$, $\{1,2,4\}$ and $\{1,3,4\}$, then $\mathcal{F}$ contains a cycle $\{2,3,4\}$.  So $\mathcal{F}$ contains at most two sets of cardinality 3. We distinguish several cases. In the first case, we suppose that all doubletons of $\mathcal{F}$ have a common element, say $1$. If there are two tripletons, their intersection is an element of $\mathcal{F}$, say $\{1,2\}$. So
$$\mathcal{F}^\ast \subset \{\{1,2\},\{1,3\},\{1,4\},\{1,2,3\},\{1,2,4\}\}.$$
Then, $\mathfrak{fr}(\mathcal{F}^\ast)\leq 1$ by application of Theorem~\ref{constantorders} to the orders $1423$ and $1324$.  If there are no two tripletons, the remaining possibility within the first case is that
$$\{2,3,4\}\in \mathcal{F}^\ast \subset \{\{1,2\},\{1,3\},\{1,4\},\{2,3,4\}\}.$$
But notice that if $\mathcal{F}$ has two of those doubletons, say $\{1,2\}$ and $\{1,3\}$, then we get a cycle $\{1,2,3\}$. So we may suppose $\mathcal{F}^\ast \subset \{\{1,2\},\{2,3,4\}\}$. Then $\mathfrak{fr}(\mathcal{F}^\ast)=0$ since it is linearly ordered by inclusion.
If the first case does not hold, and we cannot form a cycle with the doubletons, then we can suppose that the doubletons of $\mathcal{F}$ are contained in $\{\{1,2\},\{2,3\},\{3,4\}\}$. The second case is that all these three doubletons belong to $\mathcal{F}$. If we add $\{1,2,4\}$ or $\{1, 3,4\}$ we would get cycles $\{2,3,4\}$ or $\{1,2,3\}$ respectively. So
$$\mathcal{F}^\ast \subset \{\{1,2\},\{2,3\},\{3,4\},\{1,2,3\},\{2,3,4\}\},$$
and we can apply Theorem~\ref{constantorders} for the orders 1234 and 4321. The third case is that we get at most two doubletons. Since we already excluded case 1, those two doubletons are $\{1,2\}$ and $\{3,4\}$. Then, we can suppose that
$$\mathcal{F}^\ast \subset \{\{1,2\},\{3,4\},\{1,2,3\},\{1,2,4\}\}.$$
The two tripletons and $\{3,4\}$ induce a cycle $\{1,3,4\}$. So either
$$\mathcal{F}^\ast \subset \{\{1,2\},\{1,2,3\},\{1,2,4\}\}.$$ in which case we can use Theorem~\ref{constantorders} for the orders 3124 and 4213, or there is only one tripleton, so we can suppose that
$$\mathcal{F}^\ast \subset \{\{1,2\},\{3,4\},\{1,2,3\}\},$$ and then we can use the orders 1234 and 4321.
 \end{proof}

\section{Familes $\mathcal{F}[\prec]$}

If $\mathcal{F}$ is as in Theorem~\ref{constantorders},  according to Theorem~\ref{reorderings} there must exist an $n$-nestd orders  family $\mathfrak{S}$ such that $\mathcal{F} \subset \mathcal{F}_{\mathfrak{S},n}$ It is the following:
$$\mathfrak{S} = \{(t_1,\ldots,t_p) : p\leq n+2 \text{ and } t_i <_k t_k \text{ whenever }  k<i,\ k\leq n+1\}.$$ In fact, $\mathcal{F} \subset \mathcal{F}_{\mathfrak{S},n}$ because if $A = \{t\in X : t\leq_1 s_1, \ldots, t \leq_{n+1} s_{n+1}\}$ and $(t_1,\ldots,t_{n+2})\in \mathfrak{S}$ and $t_1,\ldots,t_{n+1}\in A$, then we have $t_{n+2} <_k t_{k} \leq_k s_k$, and therefore $t_{n+2}\leq_k s_k$ for all $k=1,\ldots,n+1$, and so $t_{n+2}\in A$.

We do not know if $\mathfrak{fr}(\mathcal{F}_{\mathfrak{S},n}) \leq n$ for such $\mathfrak{S}$. In what follows, we are going to discuss the simplest case when $n=1$. For this, we introduce the following notation: If $X = \{1,\ldots,m\}$ and $\prec$ is a linear order on $\{1,\ldots,m-1\}$, then the above construction for the orders $<_1 = <$ (the usual order) and $<_2 = \prec$ gives:
$$\mathfrak{S}[\prec] = \{(t_1,\ldots,t_p) : p\leq 3, t_1>t_2, t_1>t_3, t_2\succ t_1\},$$
$$\mathcal{F}[\prec] = \mathcal{F}_{\mathfrak{S}[\prec],1}.$$

For example, if $\prec$ coincides with the usual order $<$ of $\{1,\ldots,m-1\}$, then
 $$\mathcal{F}[<] = \{A\subset \{1,\ldots,m\}  : A=\{1,2,\ldots,k\}\cup \{p\} \text{ for some }0\leq k\leq p\leq m\}.$$

If $m>4$, we cannot use Theorem~\ref{constantorders} to show that $\mathfrak{fr}(\mathcal{F}[<]) \leq 1$. For suppose that there exist two linear orders $<_1$ and $<_2$ on $\{1,\ldots,m\}$ such that every element of $\mathcal{F}[<]$ is the intersection of an initial segment of $<_1$ and an initial segment of $<_2$. Since $\{1,\ldots,m-1\}$ and $\{1,\ldots,m-2,m\}$ belong to $\mathcal{F}[<]$ they must be initial segments of each of the orders. So we can suppose that $m$ is the $<_1$-maximum and $m-1$ is the $<_2$-maximum. But then, each set of the form $\{1,\ldots,k\}\cup \{m\}\in \mathcal{F}[<]$ must be an initial segment of $<_2$, and each set of the form $\{1,\ldots,k\} \cup \{m-1\}$ must be an initial segment of $<_1$. We conclude that
$$m \  <_2 \  1 \ <_2 \  2 \ <_2 \ 3 \ <_2 \ \cdots \ <_2 \ m-1,$$ $$m-1 \  <_1 \  1 \ <_1 \  2 \ <_1 \ 3 \ <_1 \ \cdots \ <_1 \ m-2 \ <_1 \ m.$$
But then $\{1,3\}\in \mathcal{F}[<]$ is not the intersection of two initial segments.\\

\begin{teo}\label{onemin}
Let $\prec$ be an order on the set $\{1,\ldots,m-1\}$. Suppose that there exists $t\in \{1,\ldots,m-1\}$ such that
$$\begin{array}{l}
t \prec t+1\prec t+2\prec t+3\prec\cdots \prec m-1,\\
 t\prec t-1 \prec t-2 \prec t-3 \prec \cdots \prec 1.
 \end{array}$$
 Then $\mathfrak{fr}(\mathcal{F}[\prec])\leq 1$.
\end{teo}

\begin{proof}
Let $S:[\omega_1]^{<\omega}\To [\omega_1]^{<\omega}$ be a function as in Lemma~\ref{specialS}.  For two subsets $A,B\subset \omega_1$ we write $A<B$ if $\alpha<\beta$ for all $\alpha\in A$ an $\beta\in B$. First we fix countable subsets $A_1,A_2,\ldots, A_t, A_{t+1} \subset \omega_1$ such that $A_1 < A_2 < \cdots < A_{t+1}$.
We are going to define $u:\{1,\ldots,m\}\To \omega_1$ inductively with respect to the order $\prec$, defining $u(t)$ in the first place, and $u(m)$ in the last place (for convenience, we declare that $s\prec m$ for all $s<m$). Along with the construction we will also define, for each $s\in\{1,\ldots,m\}$ two subsets $F_s,Z_s \subset \omega_1$ such that:
\begin{itemize}
\item $F_s$ is finite and $Z_s$ has cardinality $\omega_1$,
\item $F_r \subset F_s$ and $Z_r \supset Z_s$ whenever $r\prec s$,
\item $A_{t+1} < Z_s$ for all $s$,
\item $u(s) \in A_s\setminus F_s$ if $s\leq t$ and $u(s) = \min(Z_s)$ if $s>t$,
\item $u(r) < u(s)$ if $r<s$ (this is a consequence of the previous).

\end{itemize}

Suppose that $u(r)$, $F_t$ and $Z_t$ are defined for all $r\prec s$, and we are going to define $u(s)$, $F_s$ and $Z_s$. Let $B_s= \{u(r) : r\prec s\}$, let $F_s^- = \bigcup_{r\prec s} F_r$ and let $Z_s^- = \bigcap_{r\prec s} Z_r$. In the initial step when $s=t$, we take $B_t=F_t^-=\emptyset$ and $Z_t^-=
\{\alpha<\omega_1 : \alpha> \sup(A_{t+1})\}.$ By the $\Delta$-system lemma, we can find a subset $Z' \subset Z_s^-$ of cardinality $\omega_1$ such that $$\{S(B_s\cup \{\alpha\}) : \alpha \in Z'\}$$ is a $\Delta$-system of root $R_s$. Notice that $R_s\supset S(B_s)$. Let $T^s_\alpha = S(B_s\cup\{\alpha\})\setminus R_s$ be the tails of the $\Delta$-system. We can find a further uncountable subset $Z_s\subset Z'$ such that
$$ (\star)\ \ R_s \cup \{\min(Z_s^-)\} < T^s_\alpha \cup \{\alpha\} < T^s_\beta\cup \{\beta\} \ \ \text{ for all }  \alpha,\beta\in Z_s \text{ with }\alpha<\beta.$$ Let $F_s = F_s^- \cup R_s$. As for the choice of $u(s)$, if $s>t$ then $u(s)$ is taken as the minimum of $Z_s$, and if $s\leq t$, then we pick $u(s)\in A_s \setminus F_s$. This finishes the definition of $u$. It remains to prove that this $u$ satisfies the requirement that $Su(A) \cap Y \subset u(A)$ when $A\in \mathcal{F}[\prec]$ and $Y = \{u(1),\ldots,u(m)\}$. So, suppose that we have $A\in \mathcal{F}[\prec]$ and that $s\not\in A$ and we prove that  $u(s)\not\in Su(A)$. We consider two critical elements of $A$, $c = \max(A)$ and $a = \min(A\setminus\{c\})$, the minimum and maximum refer to usual order $<$. If $A$ is a singleton, we take $a=c$. We distinguish several cases:\\

 Case 1 $c<s$. This implies that $u(s)\not\in Su(A)$ simply because we are assuming that $S$ is as in Lemma~\ref{specialS}.\\

 Case 2: $s<c$ and $c>t$. Since $A\in \mathcal{F}[\prec]$, we must have $A\setminus \{c\} \subset \{r : r\prec s\}.$  So if $B_s=\{u(r) : r\prec s\}$ as in the inductive definition of $u(s)$, then we would have $Su(A) \subset S(B_s\cup \{u(c)\})$. So it is enough to check that $u(s)\not\in S(B\cup\{u(c)\})$.

\begin{itemize}

\item  If $s>t$, then $u(s)\in Z_s$ and since $t<s<c$, $u(c)\in Z_c \subset Z_s$. By $(\star)$ above, for $\alpha=u(s)\in Z_s$ and $\beta=u(c)\in Z_s$, we have $\alpha\not\in R_s\cup T^s_\beta = S(B_s\cup \{\beta\})$, which is what we wanted to prove.

\item If $s\leq t$ and $s\prec c$, then $u(s) \in A_s < \min(Z_s^-)$, so using $(\star)$ again for $\beta = u(c)$, we obtain that $u(s)\not\in T^s_\beta\cup \{\beta\}$. Since $u(s) \in A_s\setminus F_s$, also $u(s)\not\in R_s$. Therefore $u(s)\not\in R_s \cup T^s_\beta = S(B_s\cup\{\beta\}) $.

\item If $s\leq t$ and $c\prec s$, then $B_s \cup \{u(c)\} = B_s$, and we know that $u(s)\not\in F_s \supset B_s$.\\

\end{itemize}

Case 3: $s<c\leq t$. Since $A\in\mathcal{F}[\prec]$ and $t\prec t-1 \prec\cdots$, this implies that $A = \{a,a+1,\ldots,c\}$ and $s<a$. Remember that $F_s\supset R_s \supset S(B_s)$. In this case $u(A)\subset B_s$, so $Su(A) \subset S(B_s) \subset F_s$. Since $u(s) \in A_s\setminus F_s$, we get that $u(s) \not\in Su(A)$.

\end{proof}

\begin{coro}
$$\mathfrak{fr}\left(\{A\subset \{1,\ldots,m\}  : A=\{1,2,\ldots,k\}\cup \{p\} \text{ for some }0\leq k\leq p\leq m\}\right) = 1.$$
\end{coro}

\begin{proof}
As observed before Theorem~\ref{onemin}, the above family coincides with $\mathcal{F}[<]$ when $<$ is the usual order of $\{1,\ldots,m-1\}$, which satisfies the hypothesis of Theorem~\ref{onemin} for $t=1$.
\end{proof}

Consider the family
$$ \mathcal{F} = \mathcal{F}[5\prec 3 \prec 2 \prec 4 \prec 1] =$$ $$\{\emptyset,\{1\},\{2\},\{3\},\{4\},\{5\},\{6\}, \{1,2\},\{2,3\},\{3,4\},\{3,5\},\{5,6\},$$ $$\{1,2,3\},\{2,3,5\},\{3,5,6\}, \{2,3,5,6\}, \{1,2,3,4,5,6\}\}.$$  This is a family that does not satisfy the hypothesis of Theorem~\ref{onemin}, and for which there are essential difficulties to follow any similar argument. So we do not know if $\mathfrak{fr}(\mathcal{F}) = 1$. This is Problem 2 in the introduction.

\end{document}